\documentclass[12pt]{amsart}
\usepackage[hmargin=2.5cm,vmargin=2.5cm]{geometry}
\usepackage{amsmath}
\usepackage{amsfonts}
\usepackage{amssymb}
\usepackage{amsthm}
\usepackage[noadjust]{cite}
\usepackage{enumitem}
\usepackage{setspace}
\usepackage{amsthm}
\usepackage{graphicx}
\usepackage{float}
\usepackage{subcaption}
\usepackage{xcolor}
\usepackage{tikz}
\usepackage{pinlabel}
\usepackage{cleveref}

\AtBeginDocument{
	\def\MR#1{}
}

\pagecolor{white}

\allowdisplaybreaks

\newcommand{\Z}{\mathbb{Z}}
\newcommand{\N}{\mathbb{N}}
\newcommand{\R}{\mathbb{R}}

\newcommand{\Q}{\mathbb{Q}}

\newcommand{\F}{\mathbb{F}}

\newtheorem{thm}{Theorem}
\numberwithin{thm}{section}

\newtheorem{conj}[thm]{Conjecture}
\newtheorem{prop}[thm]{Proposition}
\newtheorem{lemma}[thm]{Lemma}
\newtheorem{cor}[thm]{Corollary}
\newtheorem{question}[thm]{Question}

\newtheorem*{namedtheorem}{\theoremname}
\newcommand{\theoremname}{testing}

\theoremstyle{definition}

\newtheorem*{nameddef}{\defname}
\newcommand{\defname}{testing}

\theoremstyle{remark}
\newtheorem{rmk}[thm]{Remark}

\begin{document}
	\title{Positive Knots and Ribbon Concordance}
	\author{Joe Boninger}
	\address{Department of Mathematics, Boston College, Chestnut Hill, MA}
	\email{boninger@bc.edu}
	\maketitle
	
	\begin{abstract}
		Ribbon concordances between knots generalize the notion of ribbon knots. Agol, building on work of Gordon, proved ribbon concordance gives a partial order on knots in $S^3$. In previous work, the author and Greene conjectured that positive knots are minimal in this ordering. In this note we prove this conjecture for a large class of positive knots, and show that a positive knot cannot be expressed as a non-trivial band sum---both results extend earlier theorems of Greene and the author for special alternating knots. In a related direction, we prove that if positive knots $K$ and $K'$ are concordant and $|\sigma(K)| \geq 2g(K) - 2$, then $K$ and $K'$ have isomorphic rational Alexander modules. This strengthens a result of Stoimenow, and gives evidence toward a conjecture that any concordance class contains at most one positive knot.
	\end{abstract}
	
	\section{Introduction}
	
	A {\em smooth concordance} between knots $K_0, K_1 \subset S^3$ is a smooth, properly embedded cylinder $C \subset S^3 \times I$ such that $C \cap (S^3 \times \{i\}) = K_i$ for $i = 0,1$. Perturbing $C$ if necessary, we assume the height function $h : C \hookrightarrow S^3 \times I \to I$ is Morse, and we say $C$ is a {\em ribbon concordance from $K_1$ to $K_0$} if $h$ has no critical points of index two. If such a concordance exists, we say $K_1$ is {\em ribbon concordant} to $K_0$ and we write $K_0 \leq K_1$. This terminology generalizes the notion of a ribbon knot, since a knot is ribbon if and only if it is ribbon concordant to the unknot.
	
	In a now-classic paper, Gordon conjectured that ribbon concordance induces a partial ordering on the set of knots \cite{gor81}. This conjecture was settled in the affirmative by Agol \cite{ago22}, and many authors have shown that ribbon concordance places strong constraints on knot invariants. To give just a few examples, if $K_0 \leq K_1$, then:
	\begin{itemize}
		\item The Alexander polynomial $\Delta_{K_0}$ of $K_0$ divides $\Delta_{K_1}$ \cite{gil84,fp20}.
		\item The genus $g(K_0)$ of $K_0$ is less than or equal to $g(K_1)$ \cite{zem19}.
		\item If $K_1$ is fibered, then $K_0$ is as well \cite{miy18, zem19}.
	\end{itemize}

	Gordon also proved torus knots are minimal under ribbon concordance \cite{gor81}. Torus knots are examples of {\em positive knots}, knots admitting a diagram in which all crossings are positive. This motivated the author and Josh Greene to conjecture:
	\begin{conj}[{\cite[Conjecture 1.6]{bogr24},\cite[Question 1.3]{tag23}}]
		\label{conj:main}
		If $K_1 \subset S^3$ is a positive knot and $K_0 \leq K_1$, then $K_0 \cong K_1$.
	\end{conj}
	\Cref{conj:main} was also posed independently by Tagami, who proved it for positive two-bridge knots \cite{tag23}. The author and Greene proved the conjecture for fibered positive knots, following a similar result of Baker and Motegi \cite{bamo17}.
	
	In this note we prove two theorems in support of \Cref{conj:main}. First, given a two-component split link $K_0 \sqcup K_1 \subset S^3$ and an embedded band $b = I \times I \subset S^3$ satisfying $b \cap K_i = I \times \{i\}$ for $i = 0,1$, we define the {\em band sum} $K_0 \#_b K_1 \subset S^3$ by
	$$
	K_0 \#_b K_1 = (K_0 \cup K_1 - I \times \partial I) \cup \partial I \times I.
	$$
	This band sum is {\em trivial} if there exists a sphere $\Sigma \subset S^3 - (K_0 \cup K_1)$ which intersects $b$ in a single arc---in this case $K_0 \#_b K_1 \cong K_0 \# K_1$, the ordinary connect sum. We say a knot is {\em band prime} if it cannot be written as a non-trivial band sum.
	
	Miyazaki proved any band sum $K_0 \#_b K_1$ is ribbon concordant to the connect sum $K_0 \# K_1$ \cite{miy98}. Thus, band sums are a natural way in which ribbon concordances arise. Additionally, if $K' \leq K$ via a ribbon concordance with only two critical points, then basic Morse theory shows $K$ is equivalent to a band sum of $K'$ and an unknot. We prove:
	
	\begin{thm}
		\label{thm:main_one}
		Positive knots are band prime.
	\end{thm}

	We also verify \Cref{conj:main} for a large class of positive knots.
	
	\begin{thm}
		\label{thm:main_two}
		Let $K \subset S^3$ be a positive knot. If the leading coefficient of $\Delta_K$ is a prime power, then $K$ is ribbon concordance minimal.
	\end{thm}

	We list \Cref{thm:main_two} second because its proof is somewhat simpler than that of \Cref{thm:main_one}. In \cite{bogr24}, the author and Greene proved Theorems \ref{thm:main_one} and \ref{thm:main_two} for all special alternating knots; since special alternating knots are positive (up to mirroring), Theorems \ref{thm:main_one} and \ref{thm:main_two} are broad extensions of those results. Additionally, while the proof that special alternating knots are band prime in \cite{bogr24} has a combinatorial flavor, our proof of \Cref{thm:main_one} is more geometric. We use a theorem of Ozawa stating any incompressible Seifert surface of a positive knot is {\em free} \cite{oza02}, and we show that such a Seifert surface cannot witness a non-trivial, genus-preserving band sum.
	
	In a related direction, we consider the following question of Gordon:
	
	\begin{question}[{\cite[Question 6.1]{gor81}}]
		\label{ques:unique}
		Does every smooth concordance class contain a unique representative which is minimal with respect to ribbon concordance?
	\end{question}
	
	Affirmative answers to this question and to \cite[Question 6.2]{gor81} would imply a generalization of the Slice-Ribbon Conjecture. \Cref{ques:unique} also seems closely related to conjectures made independently by other authors: Rudolph conjectured that each concordance class contains at most one algebraic knot \cite{rudolph76} and Baker conjectured each concordance class contains at most one fibered knot supporting the tight contact structure \cite{bak16}. Most relevant to us, Stoimenow conjectured each concordance class contains finitely many positive knots \cite{sto15}---this was verified by Baader, Dehornoy and Liechti \cite{bdl18}. Considering \Cref{conj:main} and \Cref{ques:unique}, it is natural to posit:
	
	\begin{conj}[\cite{sto15}]
		\label{conj:unique}
		Every smooth concordance class contains at most one positive knot.
	\end{conj}
	
	We credit Stoimenow since \Cref{conj:unique} seems implicit in his work. For any knot $K$, let $d(K)$ denote the degree of $\Delta_K$ when normalized to have no negative exponents. As evidence of \Cref{conj:unique}, we prove:
	
	\begin{thm}
		\label{thm:main_three}
		Let $K$ and $K'$ be (topologically or algebraically) concordant positive knots. If $K$ satisfies $|\sigma(K)| \geq d(K) - 2$, where $\sigma$ denotes the signature, then the rational Alexander modules of $K$ and $K'$ are isomorphic.
	\end{thm}

	By the {\em rational Alexander module} of $K$ we mean the cohomology ring $H^*(\bar{X};\Q)$, where $\bar{X}$ denotes the infinite cyclic cover of the exterior of $K$, viewed as a module over the group of deck transformations. \Cref{thm:main_three} strengthens a result of Stoimenow, which concluded under the above hypotheses that $K$ and $K'$ have the same Alexander polynomial \cite[Theorem 4.5]{sto15}. Additionally, although the hypothesis that
	\begin{equation}
		\label{eq:condition}
		|\sigma(K)| \geq d(K) - 2
	\end{equation}
	is somewhat restrictive, it is known that the signatures of positive knots are linearly bounded from below by their genus \cite{bdl18}. In fact, (\ref{eq:condition}) holds for all positive knots with genus less than or equal to four with the single exception of the knot $14_{45657}$ \cite[Theorem 2.4]{sto15}\cite{cogo88, sto04, sto16}.
	
	\begin{cor}
		\label{cor:genus}
		Let $K$ and $K'$ be (topologically or algebraically) concordant positive knots. If $g(K) \leq 4$, then the rational Alexander modules of $K$ and $K'$ are isomorphic.
	\end{cor}
	
	Condition (\ref{eq:condition}) also includes all special alternating knots, since these satisfy $|\sigma(K)| = d(K)$. We prove \Cref{thm:main_three} by showing that positive knots which satisfy (\ref{eq:condition}) are {\em $\Q$-anisotropic}---for a definition of $\Q$-anisotropy, see \Cref{sec:concordance} below. By a classical result of Kervaire and Gilmer, algebraically concordant knots which are $\Q$-anisotropic and admit non-singular Seifert matrices have isomorphic rational Alexander modules \cite[Proposition 4.2]{gil84} \cite{ker71}. Thus, we could remove the requirement (\ref{eq:condition}) from \Cref{thm:main_three} if we knew that:
	
	\begin{conj}
		\label{conj:two}
		Positive knots are $\Q$-anisotropic.
	\end{conj}

	By work of Gilmer, \Cref{conj:two} may be thought of as the statement that positive knots are {\em algebraically} ribbon concordance minimal \cite[Theorem 0.1]{gil84}. Scharlemann proved \Cref{conj:two} for the case of torus knots \cite[Proposition 2.3]{sch77}.
	
	\subsection{Further discussion}
	
	\Cref{sec:concordance} below contains some results on roots of Alexander polynomials which may be of independent interest---for example, we show that Alexander polynomials of positive knots have no rational roots. We also remark that our proof of \Cref{thm:main_three} extends to {\em almost positive} knots, knots which admit a diagram with one negative crossing, using results of Tagami and Stoimenow \cite{tag14}\cite[Theorem 2.3]{sto15}. It may be interesting to consider whether Theorems \ref{thm:main_one} and \ref{thm:main_two} could also be extended to almost positive knots.
	
	\subsection{Outline} In \Cref{sec:background} we recall relevant properties of postiive knots, in \Cref{sec:band_prime} we prove \Cref{thm:main_one}, and in \Cref{sec:minimal} we prove \Cref{thm:main_two}. In \Cref{sec:concordance} we discuss $\Q$-anisotropy and prove \Cref{thm:main_three} and \Cref{cor:genus}.
	
	\subsection{Acknowledgements}
	
	The author thanks Josh Greene for guidance with math and writing, and Ian Biringer for a helpful conversation about handlebodies. This material is based upon work supported by the National Science Foundation under Award No.~2202704.
	
	\section{Properties of Positive Knots}
	\label{sec:background}
	
	We gather some useful facts about positive knots. First, by results of Rudolph, the genus $g(K)$ and slice genus $g_4(K)$ of a positive knot $K$ are equal \cite{rudolph93, rudolph99}. This motivates the following lemma, well known to experts:
	
	\begin{lemma}
		\label{lem:genus}
		Let $K_0, K_1 \subset S^3$ be such that $K_0 \leq K_1$ and $K_1$ satisfies $g(K_1) = g_4(K_1)$. Then $g(K_0) = g(K_1)$.
	\end{lemma}

	In particular, the conclusion of \Cref{lem:genus} holds if $K_0 \leq K_1$ and $K_1$ is positive.

	\begin{proof}
		Since genus is non-increasing under ribbon concordance \cite{zem19} and slice genus is a concordance invariant, we have
		$$
		g_4(K_1) = g_4(K_0) \leq g(K_0) \leq g(K_1) = g_4(K_1).
		$$
	\end{proof}

	Given Rudolph's work, it is natural to ask whether minimal genus Seifert surfaces of positive knots are special in some way. A result in this direction was proven by Ozawa. A Seifert surface $S \subset S^3$ is called {\em free} if $S^3 - \nu(S)$ is a handlebody, where $\nu$ denotes a regular open neighborhood; equivalently, $S$ is free if $\pi_1(S^3 - S)$ is a free group.
	
	\begin{thm}[{\cite[Corollary 1.2]{oza02}}]
		\label{thm:ozawa}
		If $K$ is a positive knot, then every incompressible Seifert surface of $K$ is free.
	\end{thm}

	Finally, we will use the fact that positive knots are {\em pseudo-alternating} (not be confused with {\em quasi-alternating}!). The precise definition of pseudo-alternating will not be important to us (see \cite[Section 4]{mm76}), but for experts we note that pseudo-alternating links are those links which can be built from Murasugi sums of special alternating links. Positive knots are pseudo-alternating because they are homogeneous \cite{cro89}.
	
	\section{Positive Knots are Band Prime}
	\label{sec:band_prime}
	
	In this section we prove \Cref{thm:main_one}, first recalling some standard definitions from three-manifold topology. Let $Y$ be a three-manifold and $\Sigma \subset Y$ a properly embedded surface. A {\em compressing disk for} $\Sigma$ is an embedded disk $D \subset Y$ with $D \cap \Sigma = \partial D$, such that $\partial D$ does not bound a disk in $\Sigma$. Similarly, a {\em boundary-compressing disk for} $\Sigma$ is a disk $D \subset Y$ such that:
	\begin{itemize}
		\item $D \cap \Sigma \subset \partial D$
		\item $\partial D$ consists of an arc in $\partial Y$ and an arc in $\Sigma$ which is not boundary-parallel in $\Sigma$.
	\end{itemize}
	The arc $\partial D \cap \Sigma$ is called a {\em boundary-compressing arc}. The surface $\Sigma$ is called {\em compressible} (resp.~{\em boundary-compressible}) if it admits a compressing disk (resp.~boundary-compressing disk), and {\em incompressible} (resp.~{\em boundary-incompressible}) otherwise.
	
	We will need a couple lemmas on surfaces in handlebodies---the first is a classical fact.
	
	\begin{lemma}[{\cite[Example III.13]{jac80}}]
		\label{lem:bpone}
		Let $H$ be a handlebody. If $\Sigma \subset H$ is a connected surface which is incompressible and boundary-incompressible, then $\Sigma$ is a disk.
	\end{lemma}
	
	In the next lemma, by a {\em planar surface} we mean a compact surface which can be embedded in $\R^2$.
	
	\begin{lemma}
		\label{lem:bpthree}
		Let $H$ be a handlebody with boundary $F = \partial H$. Let $\Sigma \subset H$ be a properly embedded, connected planar surface such that:
		\begin{itemize}
			\item $\Sigma$ is incompressible in $H$, and 
			\item The components of $\partial \Sigma$ are separating and parallel to one another in $F$.
		\end{itemize}
		Then $\Sigma$ is either a disk or a boundary-parallel annulus.
	\end{lemma}
	
	\begin{proof}
		We will suppose $\Sigma$ has more than one boundary component (i.e.~that $\Sigma$ is not a disk) and show $\Sigma$ is a boundary-parallel annulus. By \Cref{lem:bpone}, since $\Sigma$ is incompressible it is boundary-compressible. Let $D$ be a boundary-compressing disk for $\Sigma$, so that $\partial D = \alpha \cup \alpha'$ with $\alpha' \subset F$ and $\alpha$ a properly embedded arc in $\Sigma$ which is not boundary-parallel.
		
		We first suppose the two boundary points $\partial \alpha = \partial \alpha'$ lie in distinct boundary components $\gamma$ and $\gamma'$ of $\partial \Sigma$. Since $\gamma$ and $\gamma'$ are parallel in $F$ they cobound an annulus $A \subset F$, and since they are separating $\alpha'$ lies in $A$. Let $N$ be a regular neighborhood of $A \cup D$ in $H - \Sigma$, and let
		$$
		B = N \cap \Sigma = \partial N \cap \Sigma.
		$$
		Equivalently, $B$ is a regular neighborhood of $\gamma \cup \alpha \cup \gamma'$ in $\Sigma$. Let $C = \partial N - (A \cup B)$, as in Figure \ref{fig:disk_ann}.
		
		\begin{figure}
			\labellist
			\small\hair 2pt
			\pinlabel $A$ at 75 370
			\pinlabel $B$ at 18 315
			\pinlabel $C$ at 133 12
			\pinlabel $D$ at 400 405
			\pinlabel $\gamma$ at 40 55
			\pinlabel $\gamma'$ at 502 193
			\pinlabel $\alpha$ at 435 325
			\endlabellist
			
			\includegraphics[height=5cm]{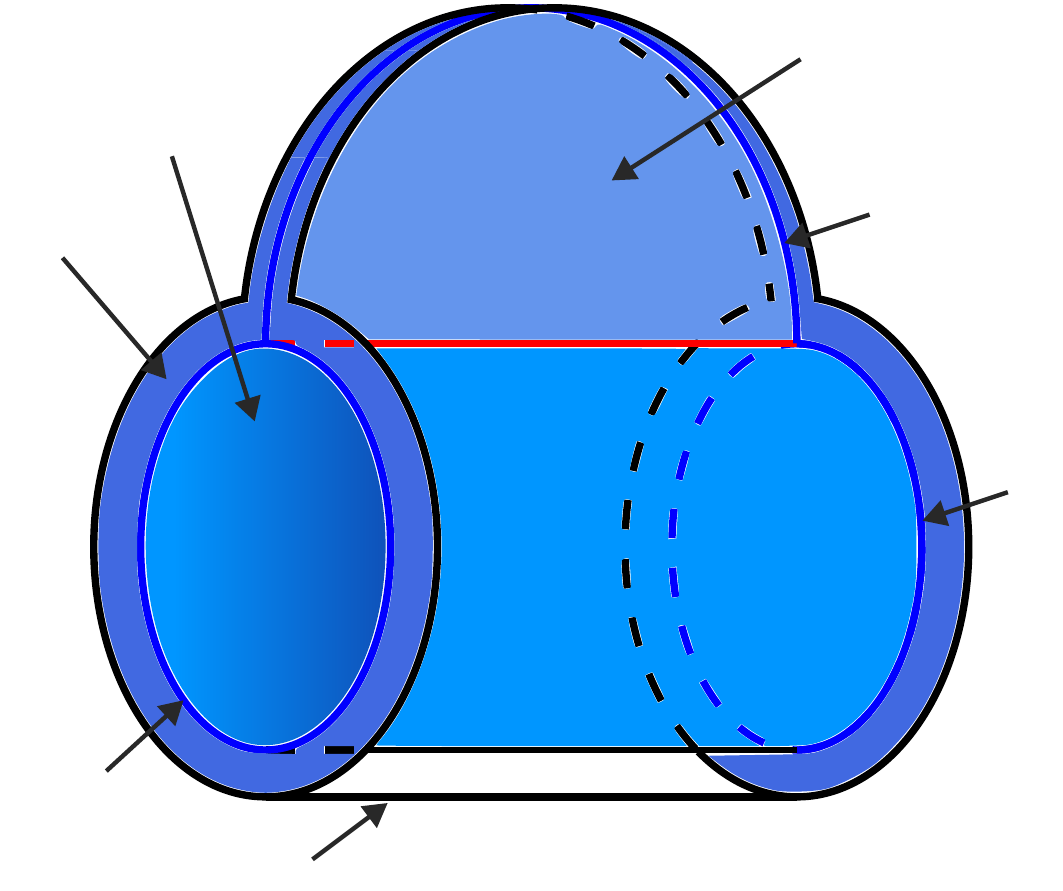}
			\caption{}
			\label{fig:disk_ann}
		\end{figure}
		
		The neighborhood $N$ is a thickened annulus, so $\partial N$ is a torus. Since $A \cup B$ is a torus with a disk removed, as the figure shows, it follows that $C$ is a disk. Now
		$$
		\partial C = \partial B \subset \Sigma,
		$$
		and the incompressibility of $\Sigma$ implies $\partial C$ bounds a disk $C'$ in $\Sigma$. Thus $\Sigma = B \cup C'$, and in this case $\Sigma$ is an annulus with boundary $\gamma \cup \gamma'$. The disks $C$ and $C'$ are parallel since they cobound a three-ball, and it follows that $\Sigma$ is boundary-parallel.
		
		We have shown that if $\Sigma$ is not an annulus, then for any boundary-compressing disk $D$ with $\partial D = \alpha \cup \alpha'$ as above, the endpoints of $\alpha \subset \Sigma$ lie in the same component of $\partial \Sigma$. The planarity of $\Sigma$ then implies each such $\alpha$ separates $\Sigma$. We therefore choose a boundary-compressing disk $D$ whose compressing arc $\alpha \subset \Sigma$ is ``outermost,'' i.e.~one of the two components of $\Sigma - \alpha$ does not contain any boundary-compressing arcs which are not isotopic to $\alpha$. Let $\Sigma'$ denote this component of $\Sigma - \alpha$.
		
		Finally, we consider the surface $\Sigma'' = \Sigma' \cup_{\alpha} D$. Since $D$ is a boundary-compressing disk $\alpha$ is not boundary-parallel in $\Sigma$, so neither $\Sigma'$ nor $\Sigma''$ is a disk. Additionally, $\Sigma''$ is incompressible: suppose, toward a contradiction, that $\Sigma''$ admits a compressing disk $D'$. Then $\partial D'$ may be isotoped into $\Sigma'$, and the incompressibility of $\Sigma$ implies $\partial D'$ bounds a disk in $\Sigma$. Thus $\Sigma$ is the union of $\Sigma'$ with a disk, which contradicts the assumption that $\alpha$ is not boundary-parallel in $\Sigma$. A similar argument, using the fact that $\alpha$ is outermost in $\Sigma$, shows $\Sigma''$ is boundary-incompressible. This contradicts \Cref{lem:bpone}, and we conclude that $\Sigma$ is an annulus as before.
	\end{proof}
	
	\Cref{thm:main_one} now follows from the more general \Cref{thm:band_prime} below. Positive knots satisfy the hypotheses of \Cref{thm:band_prime} by the discussion in \Cref{sec:background}.
	
	\begin{thm}
		\label{thm:band_prime}
		Let $K \subset S^3$ be a knot such that:
		\begin{itemize}
			\item $g(K) = g_4(K)$, and
			\item Every minimal genus Seifert surface of $K$ is free.
		\end{itemize}
		Then $K$ is band prime.
	\end{thm}
	
	\begin{proof}
		Let $K$ be a knot satisfying the given hypotheses, and suppose $K$ can be written as a band sum of knots $K_0, K_1 \subset S^3$. We will show $K$ is equivalent to the standard connect sum $K_0 \# K_1$; by an observation of Miyazaki, this occurs if and only if the band sum is trivial \cite{miy20} (cf. ~\cite{em92}).
		
		As mentioned in the introduction, Miyazaki also showed $K$ is ribbon concordant to the connect sum $K_0 \# K_1$ \cite{miy98}. Thus, by \Cref{lem:genus}, 
		$$
		g(K) = g(K_0 \# K_1) = g(K_0) + g(K_1).
		$$
		Gabai proved that a band sum preserves genus in the above sense if and only if there exists a minimal genus Seifert surface $S$ for $K$, such that $S$ is a band sum
		$$
		S = S_0 \#_b S_1
		$$
		of Seifert surfaces $S_i$ for $K_i$, $i = 0,1$ \cite{gab87}. In other words, $S$ is the result of joining the split union $S_0 \sqcup S_1$ along a band $b$. The band $b$ may be different from the band in our initial band sum representation of $K$, but this is allowed by the discussion in the first paragraph. We fix such a surface $S$, and let $\Sigma \subset S^3$ be a sphere separating the component surfaces $S_0$ and $S_1$. We choose $\Sigma$ transverse to $S$ so that the number of intersection components $|\Sigma \cap S|$ is minimal among all such spheres; then $\Sigma \cap S$ consists of a set of parallel co-cores of $b$, and $|\Sigma \cap S|$ is odd since $\Sigma$ separates the feet of $b$.
		
		Let $\nu(S)$ be a regular neighborhood of $S$, and let $H = S^3 - \nu(S)$. We claim the planar surface
		$$
		\Sigma \cap H = \Sigma - \nu(S)
		$$ 
		is incompressible in $H$. Suppose not: then $\Sigma \cap H$ admits a compressing disk $D$. The curve $\partial D$ separates $\Sigma$ into components $\Sigma_1$ and $\Sigma_2$, and each component $\Sigma_i$ contains some intersection with $S$ since $\partial D$ is essential in $\Sigma - \nu(S)$. Because $\Sigma$ separates $S_0$ and $S_1$, one of the spheres $\Sigma_0 \cup D$ or $\Sigma_1 \cup D$ does as well. Assuming the former without loss of generality, we conclude that $\Sigma_0 \cup D$ is a sphere separating $S_0$ and $S_1$ with
		$$
		|(\Sigma_0 \cup D) \cap S| < |\Sigma \cap S|.
		$$
		This contradicts the minimality of $|\Sigma \cap S|$, proving the claim.
		
		The surface $S$ has minimal genus, so $H$ is a handlebody by hypothesis. Now $\Sigma \cap H$ is an incompressible planar surface in $H$, and since $\Sigma \cap S$ consists of a set of parallel co-cores of $b$, $\partial (\Sigma \cap H)$ consists of a set of parallel curves which are separating on $\partial H$. From \Cref{lem:bpthree}, since $|\Sigma \cap S|$ is odd, we conclude that $|\Sigma \cap S| = 1$ and $\Sigma - S$ is a disk. Thus the band sum $K_0 \#_b K_1$ is trivial and $K$ is band prime.
	\end{proof}

	\begin{rmk}
		It is not true that {\em strongly quasi-positive} knots are band prime---see \cite[Section 4.1]{bamo17}---but it is true that fibered strongly quasi-positive knots are band prime by a result of Baker and Motegi \cite[Theorem 1.1]{bamo17}. In fact Baker and Motegi's argument shows that fibered strongly quasi-positive knots, like fibered positive knots, are ribbon concordance minimal.
	\end{rmk}
	
	\section{A Condition for Minimality}
	\label{sec:minimal}
	
	In this section we prove \Cref{thm:main_two}, which involves piecing together several existing results. First, the {\em lower central series} $\{\gamma_i\}_{i \geq 0}$ of a group $G$ is defined recursively by
	\begin{align*}
		\gamma_0 &= G \\
		\gamma_i &= [\gamma_{i - 1},G] \text{ for all } i > 0,
	\end{align*}
	where $[ * , *]$ indicates the commutator. The group $G$ is {\em residually nilpotent} if $\bigcap_{i = 0}^\infty \gamma_i = \{1\}$, and following Gordon \cite{gor81} we say a knot $K \subset S^3$ is {\em residually nilpotent} if the commutator subgroup of the knot group is residually nilpotent. As in the introduction, let $d(K)$ denote the degree of $\Delta_K$. Then Gordon proves:
	
	\begin{lemma}[{\cite[Lemma 3.4]{gor81}}]
		\label{lem:gordon}
		Let $K_0, K_1 \subset S^3$ be knots with $K_0 \leq K_1$. If $K_1$ is residually nilpotent and $d(K_0) = d(K_1)$, then $K_0 \cong K_1$.
	\end{lemma}

	Fibered knots are examples of residually nilpotent knots, since their commutator subgroups are free (see \cite[Chapter 5]{mks04}), but little is known in general about which knots are residually nilpotent. Murasugi and Mayland proved the following theorem:
	
	\begin{thm}[\cite{mm76}]
		\label{thm:mm}
		Let $K$ be a pseudo-alternating knot such that the leading coefficient of $\Delta_K$ is a prime power. Then $K$ is residually nilpotent.
	\end{thm}
	
	Next, we recall some background on knot Floer homology \cite{os04, ras03}. The hat version of knot Floer homology, $\widehat{HFK}$, associates a finitely generated, bigraded $\F_2$-vector space to any knot $K$:
	$$
	\widehat{HFK}(K) = \bigoplus_{i,j \in \Z} \widehat{HFK}_i(K,j).
	$$
	The $i$ and $j$ gradings are called the {\em Maslov} and {\em Alexander} gradings respectively, and the graded Euler characteristic of $\widehat{HFK}$ is the symmetrized Alexander polynomial:
	\begin{equation}
		\label{eq:euler}
		\Delta_K(t) = \sum_{i,j} (-1)^i \dim(\widehat{HFK}_i(K,j)) t^j.
	\end{equation}
	Generalizing the classical fact that $d(K) \leq 2g(K)$, knot Floer homology detects knot genus in the following sense \cite{os04b}:
	$$
	g(K) = \max \{j \mid \widehat{HFK}(K,j) \neq 0 \}.
	$$
	
	Any concordance $C \subset S^3 \times I$ between knots $K_0 \subset S^3 \times \{0\}$ and $K_1 \subset S^3 \times \{1\}$ induces a bigrading-preserving homomorphism
	$$
	C_* : \widehat{HFK}(K_0) \to \widehat{HFK}(K_1),
	$$
	 and Zemke proved that if $C$ is a ribbon concordance from $K_1$ to $K_0$, so $K_0 \leq K_1$, then the map $C_*$ is injective \cite{zem19}. Finally, we require the following theorem of Cheng, Hedden and Sarkar:

	\begin{thm}[{\cite[Corollary 1.6]{chs22}}]
		\label{thm:grading}
		If $K$ is a pseudo-alternating link, then the top Alexander grading $\widehat{HFK}(K, g(K))$ of $\widehat{HFK}(K)$ is supported in a single Maslov grading.
	\end{thm}

	\Cref{thm:main_two} now follows easily from these results and the next proposition (cf.~\cite[Propositon 1.4]{bogr24}).
	
	\begin{prop}
		\label{prop:alex_1}
		Let $K_1$ be a pseudo-alternating knot such that $g(K) = g_4(K)$, and suppose $K_0 \leq K_1$. Then $\Delta_{K_0} = \Delta_{K_1}$.
	\end{prop}
	
	\begin{proof}
		By \Cref{lem:genus}, $g(K_0) = g(K_1) = g$ for some $g \in \N$. Fix a ribbon concordance $C \subset S^3 \times I$ from $K_1$ to $K_0$, and let $C_*$ denote the induced map
		$$
		C_* : \widehat{HFK}(K_0, g) \to \widehat{HFK}(K_1, g).
		$$
		This map is injective by Zemke's result, and both groups are nonzero since the knots have the same genus. By \Cref{thm:grading} $\widehat{HFK}(K_1;g)$ is supported in a single Maslov grading, and thus $\widehat{HFK}(K_0;g)$ is as well. Consequently, (\ref{eq:euler}) implies that
		$$
		d(K_0) = 2g(K_0) = 2g(K_1) = d(K_1).
		$$
		Since $\Delta_{K_0}$ divides $\Delta_{K_1}$, we have $\Delta_{K_0} = m\Delta_{K_1}$ for some $m \in \Z$ \cite{gil84}. But
		$$
		\Delta_{K_0}(1) = \Delta_{K_1}(1) = 1,
		$$
		so $\Delta_{K_0} = \Delta_{K_1}$.
	\end{proof}

	\begin{proof}[Proof of \Cref{thm:main_two}.]
		Suppose $K_0$ and $K_1$ are knots such that $K_1$ satisfies the hypotheses of the theorem and $K_0 \leq K_1$. By \Cref{prop:alex_1}, $\Delta_{K_0} = \Delta_{K_1}$. Additionally $K_1$ is residually nilpotent by \Cref{thm:mm}, so \Cref{lem:gordon} implies $K_1 \cong K_0$.
	\end{proof}

	Two-bridge knots are residually nilpotent by a theorem of Johnson \cite{joh21}. Thus our proof of \Cref{thm:main_two} also gives an alternate proof of Tagami's theorem that positive two-bridge knots are ribbon concordance minimal, using \cite[Corollary 1.3]{joh21} in place of \Cref{thm:mm}.

	\section{Alexander Modules and $\Q$-Anisotropy}
	\label{sec:concordance}
	
	We now work toward the proofs of \Cref{thm:main_three} and \Cref{cor:genus}. We expect the following proposition is known to experts, but we have not been able to find it in the literature.
	
	\begin{prop}
		\label{prop:roots}
		Let $K \subset S^3$ be a knot such that $\Delta_K$ has a rational root $q$. Then $q = (a -1 )/a$ for some integer $a \notin \{0,1\}$. In particular, $q$ is positive.
	\end{prop}

	\begin{proof}
		Let $q = a/b$ for $a, b \in \Z$. Fix an oriented Seifert surface $S$ for $K$, and let 
		$$
		\iota_\pm : H_1(S) \to H_1(S^3 - S)
		$$
		be the maps induced by pushing curves off $S$ to the $\pm$-component of the unit normal bundle of $S$, with sign determined by the orientation. Since Alexander duality yields a canonical isomorphism $H_1(S^3 - S) \cong H_1(S)$, it makes sense to discuss the determinants of $\iota_+$ and $\iota_-$.
		
		Now 
		$$
		\Delta_K(t) = \det(\iota_+ - t\iota_-),
		$$
		and $\Delta_K(q) = 0$ implies that
		$$
		0 = \det(b\iota_+ - a\iota_-).
		$$
		Therefore there exists a non-zero vector $v \in H_1(S)$ such that
		$$
		b\iota_+(v) = a\iota_-(v).
		$$
		Dividing by a scalar if necessary, we assume $v$ is primitive, i.e.~that $v$ extends to a basis of $H_1(S)$.
		
		The intersection pairing
		$$
		\cdot : H_1(S) \times H_1(S) \to \Z
		$$
		satisfies the identity
		$$
		v_1 \cdot v_2 = \text{lk}(v_1, (\iota_+ - \iota_-)(v_2))
		$$
		for all $v_1, v_2 \in H_1(S)$, where lk indicates the linking number. Since $v$ is a primitive homology class on a once-punctured surface, $v$ is representable by a simple, closed non-separating curve on $S$ \cite{sch76,mey76}, and it follows that there exists $w \in H_1(S)$ such that $w \cdot v = 1$.
		
		We have
		$$
			1 = \text{lk}(w, (\iota_+ - \iota_-)(v)) = \text{lk}(w, \iota_+(v)) - \text{lk}(w, \iota_-(v))
		$$
		and therefore
		$$
		a\text{lk}(w, \iota_+(v)) = b\text{lk}(w, \iota_-(v)) = b(\text{lk}(w, \iota_+(v)) - 1).
		$$
		Since $a$ and $b$ are both non-zero, lk$(w, \iota_+(v)) \notin \{0,1\}$. We conclude that
		$$
		q = \frac{a}{b} = \frac{\text{lk}(w, \iota_+(v)) - 1}{\text{lk}(w, \iota_+(v))}
		$$
		as desired.
	\end{proof}
	
	\begin{cor}
		\label{cor:roots}
		If $K$ is a positive knot, then $\Delta_K$ has no rational roots.
	\end{cor}

	\begin{proof}
		Let $K$ be positive. The {\em Conway polynomial} of $K$, $\nabla_K(z) \in \Z[z^2]$, is the unique polynomial satisfying
		$$
		\nabla_K(x - x^{-1}) = \Delta_K(x^2)
		$$
		where $\Delta_K$ is the symmetrized Alexander polynomial. Cromwell proved that for a positive knot $K$, the coefficients of $\nabla_K$ are non-negative \cite[Corollary 2.1]{cro89}. Furthermore
		$$
		\nabla_K(0) = \Delta_K(1) = 1,
		$$
		so $\nabla_K$ has no real roots. It follows that $\Delta_K$ has no positive real roots, since a positive real root $q$ of $\Delta_K$ would yield a real root $\sqrt{q} - 1/\sqrt{q}$ of $\nabla_K$. Therefore, by \Cref{prop:roots}, $\Delta_K$ has no rational roots.
	\end{proof}

	Let $\bar{X}$ denote the infinite cyclic cover of the exterior of a knot $K$. For any field $\F$, an {\em invariant $\F$-isotropic subspace} of $H^1(\bar{X}; \F) \cong H^1(\bar{X}, \partial \bar{X}; \F)$ is one which is preserved by the action of deck transformations and self-annihilating with respect to the cup product
	$$
	\smile : H^1(\bar{X}; \F) \times H^1(\bar{X}, \partial \bar{X}; \F) \to H^2(\bar{X}, \partial \bar{X}; \F) \cong \F.
	$$
	The knot $K$ is called {\em $\F$-anisotropic} if $H^1(\bar{X};\F)$ does not contain a non-trivial invariant $\F$-isotropic subspace. As the introduction discusses, $\Q$-anisotropy can be used to restrict changes to the Alexander module under concordance: Kervaire and Gilmer proved that if (algebraically) concordant knots $K$ and $K'$ are $\Q$-anisotropic and admit Seifert matrices which are invertible over $\Q$, then the rational Alexander modules of $K$ and $K'$ are isomorphic \cite[Proposition 4.2]{gil84} \cite{ker71}.

	\begin{prop}
		\label{prop:q_anis}
		If a positive knot $K$ satisfies $|\sigma(K)| \geq d(K) - 2$, then $K$ is $\Q$-anisotropic.
	\end{prop}

	\begin{proof}
		Let $\bar{X}$ denote the infinite cyclic cover of the exterior of $K$, and let 
		$$
		t : H^1(\bar{X}; \Q) \to  H^1(\bar{X}; \Q) 
		$$
		be the map induced by a primitive deck transformation. Let $\Lambda \subset H^1(\bar{X}; \Q)$ be a non-trivial invariant $\Q$-isotropic subspace of $H^1(\bar{X}; \Q)$.
		
		Up to a scalar, the characteristic polynomial of $t$ coincides with $\Delta_K$. Since the cup product is skew-symmetric, it is straightforward to check that $H^1(\bar{X};\Q)$ contains a one-dimensional invariant $\Q$-isotropic subspace if and only if $\Delta_K$ has a rational root, i.e.~if and only if $t$ has a rational eigenvalue (cf.~\cite[Proposition 4.3]{gor81}). Thus, \Cref{cor:roots} implies dim$(\Lambda) \geq 2$.
		
		We now consider the Milnor form \cite{mil68} $\mu$ on $H^1(\bar{X};\Q)$, defined by
		$$
		\mu(v,w) = t(v) \smile w + t(w) \smile v
		$$
		for $v, w \in H^1(\bar{X};\Q)$. As Gordon observes and is easy to check, $\Lambda$ is also a self-annihilating subspace for $\mu$ \cite[Proposition 4.5]{gor81}. Let $V_\pm$ denote a maximal subspace of $H^1(\bar{X};\Q)$ on which $\mu$ is $\pm$-definite. Then $V_\pm \cap \Lambda = \{0\}$, so
		$$
		\dim V_\pm \leq \dim H^1(\bar{X};\Q) - \dim \Lambda \leq  \dim H^1(\bar{X};\Q) - 2 = d(K) - 2.
		$$
		It follows that
		$$
		|\sigma(K)| = |\sigma(\mu)| \leq d(K) - 4,
		$$
		and since $d(K)$ and $\sigma(K)$ are even this implies the desired inequality.
	\end{proof}

	\begin{rmk}
		It is not true that positive knots are $\R$-anisotropic, even if they satisfy the hypothesis of \Cref{prop:q_anis}: for example, the Alexander polynomial of the positive knot $10_{139}$ has a negative real root. It is also not true that strongly quasi-positive knots are $\Q$-anistropic, since there exist strongly quasi-positive knots which are topologically slice---see, for example, \cite{bfll18}.
	\end{rmk}

	\begin{proof}[Proof of \Cref{thm:main_three}]
		Let $K$ and $K'$ be concordant positive knots such that $|\sigma(K)| \geq d(K) - 2$. Then $\sigma(K) = \sigma(K')$, and since $K$ and $K'$ are positive we have
		$$
		d(K) = g_4(K) = g_4(K') = d(K').
		$$
		It follows that $|\sigma(K')| \geq d(K') - 2$, so $K$ and $K'$ are both $\Q$-anisotropic by \Cref{prop:q_anis}. Since $g(K) = d(K)$ and $g(K') = d(K')$, any Seifert matrix of a minimal genus Seifert surface for $K$ or $K'$ is invertible over $\Q$. Thus, by the result of Kervaire and Gilmer discussed before \Cref{prop:q_anis}, the rational Alexander modules of $K$ and $K'$ are isomorphic.
	\end{proof}

	\begin{proof}[Proof of \Cref{cor:genus}]
		Since the knot $14_{45657}$ is only the positive knot satisfying $g(K) = 4$ and $\sigma(K) = -4$, $14_{45657}$ is not concordant to any other positive knot. The corollary then follows from \Cref{thm:main_three} and the discussion in the introduction.
	\end{proof}
		
	\bibliography{main_bib}{}
	\bibliographystyle{amsplain}
	
\end{document}